\documentclass[11pt,a4paper,twoside]{article}
\usepackage{amsthm, amsfonts,amsmath}

\newtheorem{theorem}{Theorem}[section]
\newtheorem{corollary}{Corollary}[section]
\newtheorem{definition}{Definition}[section]
\newtheorem{example}{Example}[section]

\newtheorem{proposition}{Proposition}[section]
\newtheorem{remark}{Remark}[section]

\author{Vladimir Rovenski\footnote{Department of Mathematics, University of Haifa,
 Israel
       \newline e-mail: {\tt vrovenski@univ.haifa.ac.il}
       } }

\title{Geometry of a weak para-$f$-structure}

\begin{document}

\date{}

\maketitle

\begin{abstract}
We study the geometry of the weak almost para-$f$-structure and its satellites.
This allow us to produce totally geodesic foliations and Killing vector fields and
also to take a fresh look at the para-$f$-structure introduced by A.\,Bucki and A.\,Miernowski.
We demonstrate this by generalizing several known results on almost para-$f$-manifolds.
First, we express the covariant derivative of $f$ using a new tensor on a metric weak para-$f$-structure,
then we prove that on a weak para-${\cal K}$-manifold the characteristic vector fields are Killing and $\ker f$ defines a totally geodesic foliation.
Next, we show that a para-${\cal S}$-structure is rigid (i.e., a weak para-${\cal S}$-structure is a para-${\cal S}$-structure),
and that a metric weak para-$f$-structure with parallel tensor $f$ reduces to a weak para-${\cal C}$-structure.
We obtain corollaries for $p=1$, i.e., for a weak almost paracontact~structure.

\vskip1.5mm\noindent
\textbf{Keywords}: para-$f$-structure; distribution; totally geodesic foliation; Killing vector field

\vskip1.5mm
\noindent
\textbf{Mathematics Subject Classifications (2010)}  53C15, 53C25, 53D15
%Secondary %53C21
\end{abstract}

%\setcounter{secnumdepth}{4}
%%%%%%%%%%%%%%%%%%%%%%%%%%%%%%%%%%%%%%%%%%

\section*{Introduction}

A distribution (or a foliation, associated with integrable distribution) on a pseudo-Riemannian manifold
is \textit{totally geodesic} if any geodesic of a manifold that is tangent to the distribution at one point is tangent to it at all points.
Such foliations have the simplest extrinsic geometry of the leaves and appear in Riemannian geometry, e.g., in the theory of $\mathfrak{g}$-foliations,
as kernels of degenerate tensors, e.g., \cite{AM-1995,FP-2017}.
We are motivated by the problem of finding structures on manifolds, which lead to totally geodesic foliations and Killing vector fields,
%or distributions,
%e.g., fibrations or submersions with totally geodesic fibers,
see~\cite{fip}.
%In 1963, K.\,Yano \cite{yan} introduced
A well-known source of totally geodesic foliations
%or distributions
is
a~para-$f$-structure on a smooth manifold $M^{2n+p}$, defined
%(similarly to K.~Yano's \cite{yan} $f$-structure)
using (1,1)-tensor field $f$ satisfying $f^3 = f$ and having constant rank $2n$, see \cite{BN-1985,m1976}.
The~paracontact geometry (a counterpart to the contact geometry)
%introduced by Kaneyuki and co-workers, for example in [16],
%It~
is a higher dimensional analog of almost product ($p=0$) \cite{g1967}, and almost paracontact ($p=1$) structures  \cite{CFG-survey}.
%In the case of a compatible pseudo-Riemannian metric, we obtain para-Hermitian and para-K\"{a}hler mani\-folds.
A~para-$f$-structure with $p=2$ arises in the study of hypersurfaces in almost contact manifolds, e.g., \cite{BL-69}.
Interest in para-Sasakian manifolds is due to their connection with para-K\"{a}hler manifolds and their role in mathematical~physics.
If there exists a set of vector fields $\xi_1, \ldots , \xi_p$ with certain properties, then $M^{2n+p}$ is said to
have a para-$f$-structure with complemented frames.
In this case, the tangent bundle $TM$ splits into three complementary subbundles:
$\pm1$-eigen-distributi\-ons for $f$ composing a $2n$-dimensional distribution $f(TM)$ and a $p$-dimensional distribution $\ker f$
(the kernel of $f$).

In \cite{RWo-2}, we introduced the ``weak" metric structures that generalize
an $f$-structure and a para-$f$-structure, and allow us to take a fresh look at the classical theory.
In~\cite{Rov-arxiv}, we studied geometry of a weak $f$-structure and its satellites that are analogs of ${\mathcal K}$- ${\mathcal S}$- and ${\mathcal C}$- manifolds.
In~this paper, using a similar approach, we study geometry of a weak para-$f$-structure and its important cases
related to a pseudo-Riemannian manifold endowed with a totally geodesic foliation.
A~natu\-ral question arises: {how rich are weak para-$f$-structures compared to the classical ones}?
We~study this question for weak analogs of para-${\mathcal K}$-, para-${\mathcal S}$- and para-${\mathcal C}$- structures.
The proofs of main results use the properties of new tensors, as well as the constructions required in the classical~case.
The~theory presented here can be used to deepen our knowledge of pseudo-Riemannian geometry of
manifolds equipped with distributions.

This article consists of an introduction and five sections.
In Section~\ref{sec:1}, we discuss the properties of ``weak" metric structures generalizing some classes of para-$f$-manifolds.
In Section~\ref{sec:2} we express the covariant derivative of $f$ of a weak para-$f$-structure using a new tensor
and show that on a weak para-${\mathcal K}$-manifold the characteristic vector fields are Killing and $\ker f$ defines a totally geodesic foliation.
Also, for a weak almost para-${\mathcal C}$-structure and a weak almost para-${\mathcal S}$-structure, $\ker f$ defines a totally geodesic foliation.
In Section~\ref{sec:3a}, we apply to weak almost para-${\mathcal S}$-manifolds the tensor~$h$ and prove stability of some known results.
In Section~\ref{sec:3} we complete the result in \cite{RWo-2} and prove the rigidity theorem that a weak para-${\mathcal S}$-structure is a para-${\mathcal S}$-structure.
In Section~\ref{sec:4}, we show that a weak para-$f$-structure with parallel tensor $f$ reduces to a weak para-${\mathcal C}$-structure,
we also give an example of such a structure.

\section{Preliminaries}
\label{sec:1}

Here, we describe ``weak" metric structures generalizing certain classes of para-$f$-manifolds and discuss their properties.
A \textit{weak para-$f$-structure} on a smooth manifold $M^{\,2n+p}$ is defined by a $(1,1)$-tensor field $f$ of rank $2\,n$
and a~nonsingular $(1,1)$-tensor field $Q$ satisfying, see \cite{RWo-2},
\begin{equation}\label{E-fQ-1}
 f^3 - fQ = 0,\qquad
 %\quad {\rm rank}\,f = 2n,
% \label{2.1Q-nu}
 Q\,\xi=\xi\quad (\xi\in\ker f).
\end{equation}
If $\ker f=\{X\in TM: f(X)=0\}$ is parallelizable, then we fix vector fields $\xi_i\ (1\le i\le p)$, which span $\ker f$,
and their dual one-forms $\eta^i$. We get a~\textit{weak almost para-$f$-structure}
(a weak almost paracontact structure for $p=1$), see~\cite{RWo-2},
\begin{equation}\label{2.1}
 f^2 = Q -\sum\nolimits_{i}\eta^i\otimes\xi_i, \quad \eta^i(\xi_j)=\delta^i_j \,.
\end{equation}
%called a \textit{weak almost para-$f$-structure}.
% (or a \textit{weak f.pk-structure}).
Using \eqref{2.1} we get $f(TM)=\bigcap_{i}\ker\eta^i$ and that $f(TM)$ is $f$-invariant, i.e.,
\begin{equation}\label{2.1-D}
 {f} X\in f(TM),\quad X\in f(TM).
\end{equation}
By \eqref{2.1}-\eqref{2.1-D}, $f(TM)$ is invariant for $Q$.
A weak almost $f$-structure is called \textit{normal} if the following tensor
(known for $Q={\rm id}_{TM}$, e.g., \cite{FP-2017}) is identically~zero:
\begin{align}\label{2.6X}
 N^{\,(1)}(X,Y) = [{f},{f}](X,Y) - 2\sum\nolimits_{i} d\eta^i(X,Y)\,\xi_i .
\end{align}
The Nijenhuis torsion
%$[{f},{f}]$
of ${f}$ and the exterior derivative
%$d\eta^i$
of $\eta^i$ are given~by
\begin{align}\label{2.5}
 [{f},{f}](X,Y) & = {f}^2 [X,Y] + [{f} X, {f} Y] - {f}[{f} X,Y] - {f}[X,{f} Y],\ X,Y\in\mathfrak{X}_M , \\
\label{3.3A}
 d\eta^i(X,Y) & = \frac12\,\{X(\eta^i(Y)) - Y(\eta^i(X)) - \eta^i([X,Y])\},\quad X,Y\in\mathfrak{X}_M .
\end{align}

\begin{remark}\rm
A differential $k$-\textit{form} on a smooth manifold $M$ is a skew-symmetric tensor field
$\omega$ of  type $(0, k)$. According to the conventions of
\cite{KN-69},
\begin{eqnarray}\label{eq:extdiff}
\nonumber
 & d\omega ({X}_1, \ldots , {X}_{k+1}) = \frac1{k+1}\sum\nolimits_{\,i=1}^{k+1} (-1)^{i+1} {X}_i(\omega({X}_1, \ldots , \widehat{{X}}_i\ldots, {X}_{k+1}))\\
 & +\sum\nolimits_{\,i<j}(-1)^{i+j}\,\omega ([{X}_i, {X}_j], {X}_1, \ldots,\widehat{{X}}_i,\ldots,\widehat{{X}}_j, \ldots, {X}_{k+1}),
\end{eqnarray}
where ${X}_1,\ldots, {X}_{k+1}\in\mathfrak{X}_M$ and $\,\widehat{\cdot}\,$ denotes the
operator of omission, defines a $(k+1)$-form $d\omega$ -- the \textit{exterior differential} of $\omega$.
Thus, \eqref{eq:extdiff} with $k=1$ gives~\eqref{3.3A}.
\end{remark}

If there exists a pseudo-Riemannian metric $g$ such that
%$g(\xi_i,\xi_i)=\epsilon_i\in\{-1, 1\}$~and
\begin{align}\label{2.2}
 g({f} X,{f} Y)= -g(X,Q\,Y) +\sum\nolimits_{i}
 %\epsilon_i\,
 \eta^i(X)\,\eta^i(Y),\quad X,Y\in\mathfrak{X}_M,
\end{align}
then $({f},Q,\xi_i,\eta^i,g)$ is called a {\it metric weak para-$f$-structure},
$M({f},Q,\xi_i,\eta^i,g)$ is called a \textit{metric weak para-$f$-manifold}, and $g$ is called a \textit{compatible metric}.
 Putting $Y=\xi_i$ in \eqref{2.2} and using \eqref{E-fQ-1}, we get
  $g(X,\xi_i) = \eta^i(X)$,
thus, $f(TM)\,\bot\,\ker f$ and $\{\xi_i\}$ is an orthonormal frame of $\ker f$.

\begin{remark}\rm
According to \cite{RWo-2}, a weak almost para-$f$-structure admits a compatible pseudo-Riemannian metric if ${f}$
admits a skew-symmetric representation, i.e., for any $x\in M$ there exist a neighborhood $U_x\subset M$ and a~frame $\{e_k\}$ on $U_x$,
for which ${f}$ has a skew-symmetric matrix.
\end{remark}

The following statement is well-known for the case of $Q={\rm id}_{TM}$.

\begin{proposition}
%\label{P-6}
{\rm (a)}
For a weak almost para-$f$-structure the following hold:
\[
 {f}\,\xi_i=0,\quad \eta^i\circ{f}=0,\quad \eta^i\circ Q=\eta_i\quad (1\le i\le p),\quad [Q,\,{f}]=0.
\]
{\rm (b)}
For a metric weak almost para-$f$-structure
%$M({f},Q,\xi_i,\eta^i,g)$,
the tensor ${f}$ is skew-symmetric and the tensor $Q$ is self-adjoint, i.e.,
\begin{equation}\label{E-Q2-g}
 g({f} X, Y) = -g(X, {f} Y),\quad
 g(QX,Y)=g(X,QY).
\end{equation}
\end{proposition}

\begin{proof}
(a) By \eqref{E-fQ-1} and \eqref{2.1}, ${f}^2\xi_i=0$.
Applying \eqref{E-fQ-1} to $f\xi_i$, we get $f\xi_i=0$.
To show $\eta^i\circ{f}=0$, note that $\eta^i({f}\,\xi_i)=\eta^i(0)=0$, and, using \eqref{2.1-D}, we get $\eta^i({f} X)=0$ for $X\in f(TM)$.
Next, using \eqref{2.1} and ${f}(Q\,\xi_i) = {f}\,\xi_i=0$, we get
\begin{align*}
 {f}^3 X  = {f}({f}^2 X) =  {f}\,QX -\sum\nolimits_{i}\eta^i(X)\,{f}\xi_i =  {f}\,QX,\\
 {f}^3 X  = {f}^2({f} X) = Q\,{f} X -\sum\nolimits_{i}\eta^i({f} X)\,\xi_i = Q\,{f} X
\end{align*}
for any $X\in f(TM)$. This and $[Q,\,{f}]\,\xi_i=0$ provide $[Q,\,{f}]=Q\,{f} - {f} Q = 0$.

(b) By~\eqref{2.2}, the~restriction $Q_{|\,f(TM)}$ is self-adjoint. This and \eqref{E-fQ-1} provide (\ref{E-Q2-g}b).
For any $Y\in f(TM)$ there is $\tilde Y\in f(TM)$ such that ${f}Y=\tilde Y$.
From \eqref{2.1} and \eqref{2.2} with $X\in f(TM)$ and $\tilde Y$ we get
\begin{eqnarray*}
 g(fX,\tilde Y) = g(fX, fY) \overset{\eqref{2.2}}= -g(X, QY) \overset{\eqref{2.1}}
 %\overset{(\ref{E-Q2-g}b)}
 =  -g(X, f^2 Y) =  -g(X, f\tilde Y),
\end{eqnarray*}
and (\ref{E-Q2-g}a) follows.
\end{proof}

\begin{remark}\rm
For a weak almost para-$f$-structure, the tangent bundle
%$TM$
decomposes as
%into two complementary subbundles:
 $TM=f(TM)\oplus\ker f$, where $\ker f$ is a $p$-dimensional characte\-ristic distribution;
moreover,
%due to the symmetry of $Q$,
if we assume that the symmetric tensor $Q$ is positive definite,
then $f(TM)$ decomposes into the sum of two $n$-dimensional subbundles: $f(TM)={\mathcal D}_+\oplus{\mathcal D}_-$,
%into three complementary subbundles:
corresponding to positive and negative eigenvalues of $f$,
%where ${\mathcal D}_\pm$ are positive and negative eigen-distributions for $f$,
and in this case we get
%composing a $2n$-dimensional distribution $f(TM)$,
%and $\ker f$ is a $p$-dimensional characteristic distribution.
% $f(TM)=f(TM)$ and the kernel $\ker f$ (also called characteristic distribution).
%\[
 $TM={\mathcal D}_+\oplus{\mathcal D}_-\oplus\ker f$.
%\]
\end{remark}

Define the difference
tensor $\widetilde{Q}$ (vanishing on a para-$f$-structure) by
\[
 \tilde{Q} = Q - {\rm id}_{TM}.
\]
By the above, $\widetilde{Q}\,\xi_i=0$ and $[\tilde{Q},{f}]=0$.

%The Levi-Civita connection $\nabla$ of a pseudo-Riemannian metric $g$ is given~by
%\begin{align}\label{3.2}
% & 2\,g(\nabla_{X}Y,Z) = X(g(Y,Z)) + Y(g(X,Z)) - Z(g(X,Y))\notag\\
% & + g([X,Y],Z) +g([Z,X],Y) - g([Y,Z],X),
%\end{align}	
%and has the properties
%$X\,g(Y,Z)=g(\nabla_X\,Y,Z)+g(Y,\nabla_X\,Z)$ (metric compa\-tible)
%and
%$[X,Y]=\nabla_XY-\nabla_YX$ (without torsion).
We can rewrite \eqref{2.5} in terms of the Levi-Civita connection $\nabla$ as
\begin{align}\label{4.NN}
 [{f},{f}](X,Y) = ({f}\nabla_Y{f} - \nabla_{{f} Y}{f}) X - ({f}\nabla_X{f} - \nabla_{{f} X}{f}) Y;
\end{align}
in particular, since ${f}\,\xi_i=0$,
\begin{align}\label{4.NNxi}
 [{f},{f}](X,\xi_i)= {f}(\nabla_{\xi_i}{f})X +\nabla_{{f} X}\,\xi_i -{f}\,\nabla_{X}\,\xi_i, \quad X\in \mathfrak{X}_M .
\end{align}
%\end{remark}
The {fundamental $2$-form} $\Phi$ on $M({f},Q,\xi_i,\eta^i,g)$ is defined by
\begin{align*}
%\label{Phi-classic}
 \Phi(X,Y)=g(X,{f} Y),\quad X,Y\in\mathfrak{X}_M.
\end{align*}
Since $\eta^1\wedge\ldots\wedge\eta^p\wedge\Phi^n\ne0$,
% (hence, $M$ is orientable),
a metric weak para-${f}$-manifold is orientable.

\begin{definition}\rm
A metric weak para-$f$-structure
$({f},Q,\xi_i,\eta^i,g)$ is called a \textit{weak para-${\mathcal K}$-structure} if it is normal and the form $\Phi$ is closed, i.e., $d \Phi=0$.
 We define two subclasses of weak para-${\mathcal K}$-manifolds as follows:
\textit{weak para-${\mathcal C}$-manifolds} if $d\eta^i = 0$ for any $i$, and \textit{weak para-${\mathcal S}$-manifolds}~if
\begin{align}\label{2.3}
 d\eta^i = \Phi,\quad 1\le i\le p .
\end{align}
Omitting the normality condition, we get the following: a metric weak para-$f$-structure
%$M({f},Q,\xi_i,\eta^i,g)$
is called
(i)~a \textit{weak almost para-${\mathcal S}$-structure} if \eqref{2.3} is valid;
(ii)~a \textit{weak almost para-${\mathcal C}$-structure}
if $\Phi$ and $\eta^i$ are closed forms.
\end{definition}

For $p=1$, weak para-${\mathcal C}$- and weak para-${\mathcal S}$- manifolds reduce to weak
%(almost)
para-cosymplectic manifolds and weak
%(almost)
para-Sasakian manifolds, respectively.
%, see \cite{RWo-2}.
Recall the formulas with the Lie derivative $\pounds_{Z}$ in the $Z$-direction and $X,Y\in\mathfrak{X}_M$:
\begin{eqnarray}\label{3.3B}
 (\pounds_{Z}{f})X &  = & [Z, {f} X] - {f} [Z, X],\\
\label{3.3C}
 (\pounds_{Z}\,\eta^j)X & = & Z(\eta^j(X)) - \eta^j([Z, X]) , \\
\label{3.7}
\nonumber
 (\pounds_{Z}\,g)(X,Y) &= & Z(g(X,Y)) - g([Z, X], Y) - g(X, [Z,Y])\\
 & = & g(\nabla_{X}\,Z, Y) + g(\nabla_{Y}\,Z, X).
\end{eqnarray}
The following tensors are known in the theory of para-$f$-manifolds, e.g., \cite{FP-2017}:
%\cite{blair2010riemannian}:
\begin{align}
\label{2.7X}
 N^{\,(2)}_i(X,Y) &= (\pounds_{{f} X}\,\eta^i)Y - (\pounds_{{f} Y}\,\eta^i)X \overset{\eqref{3.3A}}= 2\,d\eta^i({f} X,Y) - 2\,d\eta^i({f} Y,X),  \\
\label{2.8X}
 N^{\,(3)}_i(X) &= (\pounds_{\xi_i}{f})X \overset{\eqref{3.3B}}= [\xi_i, {f} X] - {f} [\xi_i, X],\\
\label{2.9X}
 N^{\,(4)}_{ij}(X) &= (\pounds_{\xi_i}\,\eta^j)X \overset{\eqref{3.3C}}= \xi_i(\eta^j(X)) - \eta^j([\xi_i, X])
 = 2\,d\eta^j(\xi_i, X).
\end{align}
For $p=1$, the tensors \eqref{2.7X}--\eqref{2.9X} reduce to the following tensors on (weak) almost paracontact manifolds:
%\begin{align*}
%\label{2.7X}
 $N^{\,(2)}(X,Y) = (\pounds_{\varphi X}\,\eta)Y - (\pounds_{\varphi Y}\,\eta)X, \
%\label{2.8X}
 N^{\,(3)} = \pounds_{\xi}\,\varphi,\
%\label{2.9X}
 N^{\,(4)} = \pounds_{\xi}\,\eta$ .
%\end{align*}

\begin{remark}\rm
Let $M^{2n+p}(\varphi,Q,\xi_i,\eta^i)$ be a framed weak para-$f$-manifold.
Consider the product manifold $\bar M = M^{2n+p}\times\mathbb{R}^p$, where $\mathbb{R}^p$ is a Euclidean space
with a basis $\partial_1,\ldots,\partial_p$, and define tensor fields $\bar f$ and $\bar Q$ on $\bar M$ putting
\begin{align*}
 \bar f(X, \sum a^i\partial_i) = (fX -\sum a^i\xi_i, \sum \eta^j(X)\partial_j),
 \quad
 \bar Q(X, \sum a^i\partial_i) = (QX, \sum a^i\partial_i) .
\end{align*}
Hence, $\bar f(X,0)=(fX,0)$, $\bar Q(X,0)=(QX,0)$ for $X\in\ker f$,
$\bar f(\xi_i,0)=(0,\partial_i)$, $\bar Q(\xi_i,0)=(\xi_i,0)$ and
$\bar f(0,\partial_i)=(-\xi_i,0)$, $\bar Q(0,\partial_i)=(0,\partial_i)$.
Then it is easy to verify that $\bar f^{\,2}=-\bar Q$. The tensors $N^{\,(i)}\ (i=1,2,3,4)$ appear when we use
the integrability condition $[\bar f, \bar f]=0$ of $\bar f$ to express the normality condition
%$N^{\,(1)}=0$
of a weak almost para-$f$-structure.
\end{remark}

\section{The geometry of a metric weak para-$f$-structure}
\label{sec:2}

Here, we study the geometry of the characteristic distribution $\ker f$,
supplement the sequence of tensors \eqref{2.6X} and \eqref{2.7X}--\eqref{2.9X}
with a new tensor $N^{\,(5)}$ and calculate the covariant derivative of $f$ on a metric weak para-$f$-structure.

A distribution ${\mathcal D}\subset TM$ is \textit{totally geodesic} if and only if
its second fundamental form vanishes, i.e., $\nabla_X Y+\nabla_Y X\in{\mathcal D}$ for any vector fields $X,Y\in{\mathcal D}$ --
%%%%%%%%%%%%%
this is the case when {any geodesic of $M$ that is tangent to ${\mathcal D}$ at one point is tangent to ${\mathcal D}$ at all its points}.
%e.g., \cite{KN-69}.
Any integrable and totally geodesic distribution determines a totally geodesic foliation.
A foliation, whose orthogonal distribution is totally geodesic, is said to be a Riemannian foliation.
For example, a foliation is Riemannian if it is invariant under transformations (isometries) generated by Killing vector fields.
Note that $X = X^\top + X^\bot$, where $X^\top$ is the projection of the vector $X\in TM$ onto $f(TM)$,
and $X^\bot = \sum\nolimits_{i}\eta^i(X)\,\xi_i$.
%\[
% X^\top = X - \sum\nolimits_{i}\eta^i(X)\,\xi_i .
%\]
The next statement generalizes
%result in \cite{BN-1985}, see also
\cite[Proposition~3]{FP-2017},
%\cite[Theorem~6.1]{blair2010riemannian},
i.e., $Q={\rm id}_{ TM}$.

\begin{proposition}\label{thm6.1}
Let a metric weak para-$f$-structure be normal. Then $N^{\,(3)}_i$ and $N^{\,(4)}_{ij}$ vanish~and
\begin{align}\label{3.1KK}
 N^{\,(2)}_i(X,Y) =\eta^i([\widetilde{Q} X,\,{f} Y]);
\end{align}
moreover, the characteristic distribution $\ker f$ is totally geodesic.
\end{proposition}

\begin{proof}
%Claim: $N^{\,(4)}_{ij}=0$.
Assume $N^{\,(1)}(X,Y)=0$ for any $X,Y\in TM$. Taking $\xi_i$ instead of $Y$ and using the formula of Nijenhuis tensor \eqref{2.5}, we~get
\begin{eqnarray}\label{3.11}
 0 & =& [{f},{f}](X,\xi_i) - 2\sum\nolimits_{j} d\eta^j(X,\xi_i)\,\xi_j \notag\\
 %\label{3.12} &
 & =& {f}^2[X,\xi_i] - {f}[{f} X,\xi_i] - 2\sum\nolimits_{j} d\eta^j(X,\xi_i)\,\xi_j.
\end{eqnarray}
For the scalar product of \eqref{3.11} with $\xi_j$, using
%skew-symmetry of ${f}$ and
${f}\,\xi_i=0$, we~get
\begin{align}\label{3.11A}
 d\eta^j(\xi_i,\,\cdot)=0;
 %\quad ({\rm or},\ \iota_{\xi_i} d\eta^j=0);
\end{align}
hence, $N^{\,(4)}_{ij}=0$, see \eqref{2.9X}.
%Claim: $N^{\,(3)}_i=0$.
Next, combining \eqref{3.11} and \eqref{3.11A}, we get
\begin{align*}
 %\label{3.13}
 0 = [{f},{f}](X,\xi_i) = {f}^2[X,\xi_i] - {f}[{f} X,\xi_i] = {f}\,(\pounds_{\xi_i}{f})X .
\end{align*}
Applying ${f}$ and using \eqref{2.1} and $\eta^i\circ{f}=0$, we achieve
\begin{eqnarray}\label{3.14}
\nonumber
 0 & = {f}^2 (\pounds_{\xi_i}{f})X
  = Q(\pounds_{\xi_i}{f})X - \sum\nolimits_{j}\eta^j((\pounds_{\xi_i}{f})X)\,\xi_j \\
 & =  Q(\pounds_{\xi_i}{f})X - \sum\nolimits_{j}\eta^j([\xi_i,{f} X])\,\xi_j.
\end{eqnarray}
Further, \eqref{3.11A} and \eqref{3.3A} yield
\begin{align}\label{3.11B}
	0=2\,d\eta^j({f} X, \xi_i)
	=({f} X)(\eta^j(\xi_i)) - \xi_i(\eta^j({f} X)) - \eta^j([{f} X, \xi_i])
	=\eta^j([\xi_i, {f} X]).
\end{align}
Since $Q$ is non-singular, from \eqref{3.14}--\eqref{3.11B} we get $\pounds_{\xi_i}{f}=0$, i.e, $N^{\,(3)}_i=0$, see~\eqref{2.8X}.
%Claim: $N^{\,(2)}_i(X,Y)=\eta^i([QX-X,{f} Y])$.
 Replacing $X$ by ${f} X$ in our assumption $N^{\,(1)}=0$ and using \eqref{2.5} and \eqref{3.3A}, we get
\begin{align}\label{2.6}
 0 &= g([{f},{f}]({f} X,Y) - 2\sum\nolimits_{j} d\eta^j({f} X,Y)\,\xi_j,\ \xi_i) \notag\\
 &= g([{f}^2 X,{f} Y],\xi_i) - ({f} X)(\eta^i(Y)) + \eta^i([{f} X,Y]) ,\quad 1\le i\le p.
\end{align}
Using \eqref{2.1} and
$[{f} Y, \eta^j(X) \xi_i] = ({f} Y)(\eta^j(X)) \xi_i + \eta^j(X)[{f} Y, \xi_i]$, we rewrite \eqref{2.6}~as
\begin{equation*}
%\label{2.8}
 0 = \eta^i([QX, {f} Y]) -\sum \eta^j(X)\,\eta^i([\xi_j, {f} Y])
  + {f} Y(\eta^i(X)) - {f} X(\eta^i(Y)) + \eta^i([{f} X,Y]).
\end{equation*}
Since \eqref{3.11B} gives $\eta^i([{f} Y, \xi_j])=0$, the above equation becomes
%So, \eqref{2.8} becomes
\begin{align}\label{2.9}
 \eta^i([QX, {f} Y]) + ({f} Y)(\eta^i(X)) - ({f} X)(\eta^i(Y)) + \eta^i([{f} X,Y]) = 0.
\end{align}
Finally, combining \eqref{2.9} with \eqref{2.7X}, we get \eqref{3.1KK}.
%\begin{align*}
%%\label{3.1KK}
% N^{\,(2)}_i(X,Y) =\eta^i([\widetilde{Q}X,\,{f} Y]),\quad 1\le i\le p,
%\end{align*}
%from which and $X = X^\top +\sum\nolimits_{i}\eta^i(X)\,\xi_i$ the expression \eqref{3.1KK} of $N^{\,(2)}_i$ follows.
%Since $N^{\,(1)}=0$, by Proposition~\ref{thm6.1} we get $N^{\,(4)}_{ij}=0$, i.e., \eqref{3.11A}.
Using the identity
\begin{align}\label{3.Ld}
 \pounds_{\xi_i}=\iota_{{\xi_i}}\,d + d\,\iota_{{\xi_i}},
\end{align}
from \eqref{3.11A} and $\eta^i(\xi_j)=\delta^i_j$ we obtain
%\begin{align*}
%\label{2.Q2}
 $\pounds_{\xi_i}\,\eta^j = d (\eta^j(\xi_i)) + \iota_{\xi_i}\, d\eta^j = 0$.
%\end{align*}
On the other hand, by \eqref{3.3C} we have
\[
 (\pounds_{\xi_i}\,\eta^j)X= g(X,\nabla_{\xi_i}\,\xi_j)+g(\nabla_{X}\,\xi_i,\,\xi_j),\quad
 X\in\mathfrak{X}_M.
\]
Symmetrizing this and using $\pounds_{\xi_i}\,\eta^j =0$ and $g(\xi_i,\, \xi_j)=\delta_{ij}$ yield
\begin{align}\label{3.30}
 \nabla_{\xi_i}\,\xi_j+\nabla_{\xi_j}\,\xi_i =0,
\end{align}
thus, the distribution $\ker f$ is totally geodesic.
\end{proof}

Recall the co-boundary formula for exterior derivative $d$ on a $2$-form $\Phi$,
\begin{eqnarray}\label{3.3}
\nonumber
 d\Phi(X,Y,Z) & =& \frac{1}{3}\,\big\{X\,\Phi(Y,Z) + Y\,\Phi(Z,X) + Z\,\Phi(X,Y) \\
 && -\Phi([X,Y],Z) - \Phi([Z,X],Y) - \Phi([Y,Z],X)\big\}.
\end{eqnarray}
By direct calculation we get the following:
\begin{align}\label{3.9A}
 (\pounds_{\xi_i}\,\Phi)(X,Y) = (\pounds_{\xi_i}\,g)(X, {f}Y) + g(X,(\pounds_{\xi_i}{f})Y) .
\end{align}

The following result generalizes \cite[Proposition~4]{FP-2017}.

\begin{theorem}\label{C-K}
On a weak para-${\mathcal K}$-manifold the vector fields $\xi_1,\ldots,\xi_p$ are Killing and
\begin{align}\label{6.1e}
 \nabla_{\xi_i}\,\xi_j  = 0,\quad 1\le i,j\le p ;
\end{align}
%$\nabla_{\xi_k}\,\xi_j=0$;
thus,
%the cha\-racteristic distribution
$\ker f$ is integrable and defines a totally geodesic Riemannian foliation with flat leaves.
\end{theorem}

\begin{proof}
By Proposition~\ref{thm6.1}, the distribution $\ker f$ is totally geodesic, see \eqref{3.30}, and $N^{\,(3)}_i=\pounds_{\xi_i}{f}=0$.
Using $\iota_{{\xi_i}}\Phi=0$ and condition $d\Phi=0$ in the identity \eqref{3.Ld},
we get $\pounds_{\xi_i}\Phi=0$. Thus, from \eqref{3.9A} we obtain $(\pounds_{\xi_i}\,g)(X, {f}Y)=0$.
To show $\pounds_{\xi_i}\,g=0$, we will examine $(\pounds_{\xi_i}\,g)(fX, \xi_j)$ and $(\pounds_{\xi_i}\,g)(\xi_k, \xi_j)$.
Using $\pounds_{\xi_i}\,\eta^j =0$,
we get
\[
 (\pounds_{\xi_i}\,g)(fX, \xi_j)=(\pounds_{\xi_i}\,\eta^j)fX -g(fX, [\xi_i,\xi_j])=-g(fX, [\xi_i,\xi_j])=0.
\]
Using \eqref{3.30}, we get
%\[
 $(\pounds_{\xi_i}\,g)(\xi_k, \xi_j)= -g(\xi_i, \nabla_{\xi_k}\,\xi_j+\nabla_{\xi_j}\,\xi_k) = 0$.
%\]
Thus, $\xi_i$ is a Killing vector field, i.e., $\pounds_{\xi_i} g=0$.
By $d\Phi(X,\xi_i,\xi_j)=0$ and \eqref{3.3} we obtain $g([\xi_i,\xi_j], fX)=0$, i.e., $\ker f$ is integrable.
From this and \eqref{3.30} we get $\nabla_{\xi_k}\,\xi_j=0$; thus, the sectional curvature is $K(\xi_i,\xi_j)=0$.
\end{proof}

\begin{theorem}\label{thm6.2}
For a weak almost para-${\mathcal S}$-structure,
% $({f},Q,\xi_i,\eta^i,g)$,
we get $N^{\,(2)}_i=N^{\,(4)}_{ij}=0$ and
\begin{equation}\label{E-N1}
 (N^{\,(1)}(X,Y))^\bot = 2\,g(X, f\widetilde{Q} Y)\,\bar\xi \,;
\end{equation}
moreover, $N^{\,(3)}_i$ vanishes if and only if $\,\xi_i$ is a Killing vector field.
\end{theorem}

\begin{proof} Applying \eqref{2.3} in \eqref{2.7X} and using skew-symmetry of ${f}$ we get $N^{\,(2)}_i=0$.
Equation \eqref{2.3} with $Y=\xi_i$ yields $d\eta^j(X,\xi_i)=g(X,{f}\,\xi_i)=0$ for any $X\in\mathfrak{X}_M$; thus, we get \eqref{3.11A},
i.e., $N^{\,(4)}_{ij}=0$.
Using \eqref{2.3} and
\[
 g([f,f](X,Y), \xi_i) = g([fX,fY], \xi_i) = -2\,d\eta^i(fX,fY) = -2\,\Phi(fX, fY)
\]
for all $i$, we also calculate
\begin{eqnarray*}
 & \frac12\,g(N^{\,(1)}(X,Y), \xi_i) = -d\eta^i(fX,fY) - g(\sum\nolimits_{j}d\eta^j(X,Y)\,\xi_j, \xi_i) \\
 &= -\Phi(fX, fY) -\Phi(X, Y) = g(X, (f^3-f)Y) = g(X, \widetilde{Q} f Y),
\end{eqnarray*}
that proves \eqref{E-N1}.
Next, invoking \eqref{2.3} in the equality
%\eqref{3.8}
\begin{align*}
%\label{3.8}
 (\pounds_{\xi_i}\,d\eta^j)(X,Y) = \xi_i(d\eta^j(X,Y)) - d\eta^j([\xi_i,X], Y) - d\eta^j(X,[\xi_i,Y]),
\end{align*}
and using \eqref{3.7}, we obtain for all $i,j$
\begin{align}\label{3.9}
 (\pounds_{\xi_i}\,d\eta^j)(X,Y) = (\pounds_{\xi_i}\,g)(X, {f}Y) + g(X,(\pounds_{\xi_i}{f})Y).
\end{align}
Since $\pounds_V=\iota_{V}\circ d+d\circ\iota_{V}$, the exterior derivative $d$ commutes with the Lie-derivative, i.e., $d\circ\pounds_V = \pounds_V\circ d$, and as in the proof of Theorem~\ref{C-K}, we get that $d\eta^i$ is invariant under the action of $\xi_i$, i.e., $\pounds_{\xi_i}\,d\eta^j=0$.
Therefore, \eqref{3.9} implies that $\xi_i$ is a Killing vector field if and only if $N^{\,(3)}_i=0$.
\end{proof}

\begin{theorem}\label{thm6.2C}
For a weak almost para-${\mathcal C}$-structure, we get $N^{\,(2)}_i=N^{\,(4)}_{ij}=0$, $N^{\,(1)}=[{f},{f}]$, and
\eqref{6.1e};
thus, the distribution $\ker f$ is tangent to a totally geodesic foliation with the sectional curvature $K(\xi_i,\xi_j)=0$.
Moreover, $N^{\,(3)}_i=0$ if and only if $\,\xi_i$ is a Killing vector~field.
\end{theorem}

\begin{proof}
By \eqref{2.7X} and \eqref{2.9X} and since $d\eta^i=0$, the tensors $N^{\,(2)}_i$ and $N^{\,(4)}_{ij}$ vanish on a weak almost para-${\mathcal C}$-structure.
Moreover, by \eqref{2.6X} and \eqref{3.9}, respectively, the tensor $N^{\,(1)}$ coincides with $[f,f]$,
and $N^{\,(3)}_i=\pounds_{\xi_i}{f}\ (1\le i\le p)$ vanish if and only if each $\xi_i$ is a Killing~vector.
From the equalities
\begin{align*}
%\label{6.2}
 3\,d\Phi(X,\xi_i,\xi_j)  = g([\xi_i,\xi_j], fX), \qquad
 2\,d\eta^k(\xi_j, \xi_i) = g([\xi_i,\xi_j],\xi_k)
\end{align*}
and conditions $d\Phi=0$ and $d\eta^i=0$ we obtain
\begin{align}\label{6.1d}
 [\xi_i, \xi_j] & = 0,\quad 1\le i,j\le p .
\end{align}
Next, from $d\eta^i=0$ and the equality
\[
 2\,d\eta^i(\xi_j,X)+2\,d\eta^j(\xi_i,X) = g(\nabla_{\xi_i}\,\xi_j+\nabla_{\xi_j}\,\xi_i, X)
\]
we obtain \eqref{3.30}: $\nabla_{\xi_i}\,\xi_j+\nabla_{\xi_j}\,\xi_i=0$. From this and \eqref{6.1d} we get \eqref{6.1e}.
\end{proof}

We will express $\nabla_{X}{f}$ using a new tensor on a metric weak para-$f$-structure.
The following assertion
%plays a key role in the paper and
generalizes \cite[Proposition~1]{FP-2017}.

\begin{proposition}\label{lem6.1}
For a metric weak para-$f$-structure
%$({f},Q,\xi_i,\eta^i,g)$,
we get
\begin{eqnarray}\label{3.1}
 & 2\,g((\nabla_{X}{f})Y,Z) = -3\,d\Phi(X,{f} Y,{f} Z) - 3\, d\Phi(X,Y,Z) - g(N^{\,(1)}(Y,Z),{f} X)\notag\\
 & +\sum\nolimits_{i}\big( N^{\,(2)}_i(Y,Z)\,\eta^i(X) + 2\,d\eta^i({f} Y,X)\,\eta^i(Z) - 2\,d\eta^i({f} Z,X)\,\eta^i(Y)\big)\notag\\
 & + N^{\,(5)}(X,Y,Z),
\end{eqnarray}
where a skew-symmetric w.r.t. $Y$ and $Z$ tensor $N^{\,(5)}(X,Y,Z)$ is defined by
\begin{eqnarray*}
  N^{\,(5)}(X,Y,Z) &=& ({f} Z)\,(g(X, \widetilde{Q}Y)) -({f} Y)\,(g(X, \widetilde{Q}Z)) +g([X, {f} Z], \widetilde{Q}Y)\\
 &&-\,g([X,{f} Y], \widetilde{Q}Z) + g([Y,{f} Z] -[Z, {f} Y] - {f}[Y,Z],\ \widetilde{Q} X).
\end{eqnarray*}
\end{proposition}

\begin{proof}
Using
%\eqref{3.2} and
the skew-symmetry of ${f}$, one can compute
\begin{eqnarray}\label{3.4}
 & 2\,g((\nabla_{X}{f})Y,Z) = 2\,g(\nabla_{X}({f} Y),Z) + 2\,g( \nabla_{X}Y,{f} Z) \notag\\
 & = X\,g({f} Y,Z) + ({f} Y)\,g(X,Z) - Z\,g(X,{f} Y) \notag\\
 & +\, g([X,{f} Y],Z) +g([Z,X],{f} Y) - g([{f} Y,Z],X) \notag\\
 & +\, X\,g(Y,{f} Z) + Y\,g(X,{f} Z) - ({f} Z)\,g(X,Y)\notag\\
 & +\, g([X,Y],{f} Z) + g([{f} Z,X],Y) - g([Y,{f} Z],X).
\end{eqnarray}
Using \eqref{2.2}, we obtain
\begin{align}\label{XZ}
\notag
 g(X,Z) &= -\Phi({f} X, Z) -g(X,\widetilde{Q} Z) +\sum\nolimits_{i}\big(\eta^i(X)\,\eta^i(Z) +\eta^i(X)\,\eta^i(\widetilde{Q} Z)\big)\\
 &= -\Phi({f} X, Z) + \sum\nolimits_{i}\eta^i(X)\,\eta^i(Z)  - g(X, \widetilde{Q}Z).
\end{align}
Thus, and in view of the skew-symmetry of ${f}$ and applying \eqref{XZ} six times, \eqref{3.4} can be written~as
\begin{align*}
%\label{3.5}
& 2\,g((\nabla_{X}{f})Y,Z) = X\,\Phi(Y, Z)
+({f} Y)\,\big(-\Phi({f} X, {Z})+\sum\nolimits_{i}\eta^i(X)\,\eta^i(Z) \big) \\
& - ({f} Y)\,g(X,\widetilde{Q}Z) - Z\,\Phi(X,Y) \\
& +\Phi([X,{f} Y],{f} {Z}) + \sum\nolimits_{i}\eta^i([X,{f} Y])\eta^i(Z) - g([X,{f} Y],\widetilde{Q}Z) +\Phi([Z,X],Y) \notag\\
& -\Phi([{f} Y,Z],{f} {X}) - \sum\nolimits_{i}\eta^i([{f} Y,Z])\,\eta^i(X) + g([{f} Y, Z], \widetilde{Q}X) + X\,\Phi(Y,Z) \\
& +Y\,\Phi(X,Z) - ({f} Z)\,\big(-\Phi({f} X, {Y}) + \sum\nolimits_{i}\eta^i(X)\,\eta^i(Y)\big) + ({f} Z) g(X, \widetilde{Q}Y) \\
& +\Phi([X,Y],Z) + g({f}[-{f} Z,X],{f} {Y}) + \sum\nolimits_{i}\eta^i([{f} Z,X])\eta^i(Y) - g([{f} Z,X],\widetilde{Q}Y)\\
& +g({f}[Y,{f} Z],{f} {X}) - \sum\nolimits_{i}\eta^i([Y,{f} Z])\,\eta^i(X) + g([Y,{f} Z], \widetilde{Q}X) .
\end{align*}
We also have
\begin{eqnarray*}
%\label{Eq-N1}\nonumber
  g(N^{\,(1)}(Y,Z),{f} X) = g({f}^2 [Y,Z] + [{f} Y, {f} Z] - {f}[{f} Y,Z] - {f}[Y,{f} Z], {f} X)\\
  = - g({f}[Y,Z], \widetilde{Q} X) + g([{f} Y, {f} Z] - {f}[{f} Y,Z] - {f}[Y,{f} Z] - [Y,Z], {f} X).
\end{eqnarray*}
From this and \eqref{3.3} we get the required result.
\end{proof}

\begin{remark}\rm
%The new tensor $N^{\,(5)}$ naturally supplements the traditional sequence of tensors \eqref{2.6X} and \eqref{2.7X}--\eqref{2.9X}.
For particular values of the tensor $N^{\,(5)}$ we get
\begin{eqnarray}\label{KK}
\nonumber
 N^{\,(5)}(X,\xi_i,Z) & = & - N^{\,(5)}(X, Z, \xi_i) = g( N^{\,(3)}_i(Z),\, \widetilde{Q} X),\\
%\label{KK2}
\nonumber
 N^{\,(5)}(\xi_i,Y,Z) &=& g([\xi_i, {f} Z], \widetilde{Q}Y) -g([\xi_i,{f} Y], \widetilde{Q}Z),\\
 N^{\,(5)}(\xi_i,Y,\xi_j) &=& N^{\,(5)}(\xi_i,\xi_j, Y) =0.
\end{eqnarray}
\end{remark}

We will discuss the meaning of $\nabla_{X}{f}$ for weak almost para-${\mathcal S}$- and weak para-${\mathcal K}$- structures.
The following corollary of Proposition~\ref{lem6.1} and Theorem~\ref{thm6.2}
generalizes well-known results with $Q={\rm id}_{TM}$.

\begin{corollary}\label{cor3.1}
For a weak almost para-${\mathcal S}$-structure we get
\begin{align}\label{3.1A}
\nonumber
 2\,g((\nabla_{X}{f})Y,Z) & = - g(N^{\,(1)}(Y,Z),{f} X) +2\,g(fX,fY)\,\bar\eta(Z) \\
 & -2\,g(fX,fZ)\,\bar\eta(Y) + N^{\,(5)}(X,Y,Z),
% \quad 1\le i\le p,
\end{align}
where $\bar\eta=\sum\nolimits_{i}\eta^i$.
In particular, taking $x=\xi_i$ and then $Y=\xi_j$ in \eqref{3.1A}, we get
\begin{align}\label{3.1AA}
%\nonumber
 2\,g((\nabla_{\xi_i}{f})Y,Z) &= N^{\,(5)}(\xi_i,Y,Z) ,\quad 1\le i\le p,
%\label{2-xi}
% \nabla_{\xi_i}\,\xi_j & = 0,\qquad 1\le i,j\le p.
\end{align}
and \eqref{6.1e}; thus, the characteristic distribution
%on a weak almost para-${\mathcal S}$-manifold
is tangent to a totally geodesic foliation with flat leaves.
%the sectional curvature $K(\xi_i,\xi_j)=0$.
\end{corollary}

\begin{proof}
According to Theorem~\ref{thm6.2}, for a weak almost para-${\mathcal S}$-structure we have
%\[
 $d\eta^i = \Phi$ and $N^{\,(2)}_i= N^{\,(4)}_{ij}=0$.
%\]
Thus, invoking \eqref{2.3} and using Theorem~\ref{thm6.2} in \eqref{3.1}, we get \eqref{3.1A}.
From \eqref{3.1AA} with $Y=\xi_j$ we get $g(f\nabla_{\xi_i}\,\xi_j, Z)=0$,
thus $\nabla_{\xi_i}\,\xi_j\in\ker f$.~Also,
\[
 \eta^k([\xi_i,\xi_j])= -2\,d\eta^k(\xi_i,\xi_j) =-2\,g(\xi_i, f\xi_j)=0;
\]
hence, $[\xi_i,\xi_j]=0$, i.e., $\nabla_{\xi_i}\,\xi_j=\nabla_{\xi_j}\,\xi_i$.
Finally, from $g(\xi_j,\xi_k)=\delta_{jk}$, using the covariant derivative with respect to $\xi_i$ and the above equality,
we get $\nabla_{\xi_i}\,\xi_j\in f(TM)$. This together with $\nabla_{\xi_i}\,\xi_j\in\ker f$ proves \eqref{6.1e}.
\end{proof}

\section{The tensor field $h$}
\label{sec:3a}

Here, we apply for a weak almost para-${\mathcal S}$-manifold the tensor field $h=(h_1,\ldots,h_p)$, where
%\begin{align*}
%\label{4.1}
 $h_i=\frac{1}{2}\, N^{\,(3)}_i = \frac{1}{2}\,\pounds_{\xi_i}{f}$ .
%\end{align*}
By Theorem~\ref{thm6.2}, $h_i=0$ if and only if $\xi_i$ is a Killing field.
First, we calculate
\begin{align}\label{4.2}
\nonumber
 & (\pounds_{\xi_i}{f})X \overset{\eqref{3.3B}} = \nabla_{\xi_i}({f} X) - \nabla_{{f} X}\,\xi_i - {f}(\nabla_{\xi_i}X - \nabla_{X}\,\xi_i) \\
 &\ = (\nabla_{\xi_i}{f})X - \nabla_{{f} X}\,\xi_i + {f}\nabla_X\,\xi_i.
\end{align}
For $X=\xi_i$ in \eqref{4.2}, using
$g((\nabla_{\xi_i}{f})\,\xi_j,Z)=\frac12 N^{\,(5)}(\xi_i,\xi_j,Z)=0$, see \eqref{3.1AA},
and $\nabla_{\xi_i}\,\xi_j=0$, see Corollary~\ref{cor3.1}, we get
\begin{align}\label{4.2b}
 h_i\,\xi_j = 0.
\end{align}
The following result generalizes the fact that for an almost para-${\mathcal S}$-structure, each tensor $h_i$ is self-adjoint and commutes with ${f}$.

\begin{proposition}
%\label{L3.1}
For a weak almost para-${\mathcal S}$-structure, the tensor $h_i$ and its conjugate $h_i^*$ satisfy
\begin{eqnarray}\label{E-31}
 g((h_i-h_i^*)X, Y) &=& \frac{1}{2}\,N^{\,(5)}(\xi_i, X, Y),\\
 \label{E-30b}
 \nabla\,\xi_i &=&
 %Q^{-1} {f} (h^*_i - Q) =
 Q^{-1} {f}\, h^*_i - f , \\
% g(Q\,\nabla_{X}\,\xi_i, Z) &=& -g(({f}(Q+h^*_i))X, Z) ,\\
% g(({f}+h_i{f}) Z,QX) -g(h_i f Z, \,\widetilde{Q} X) , \\
 \label{E-31A}
 h_i{f}+{f}\, h_i &=& -\frac12\,\pounds_{\xi_i}\widetilde{Q}.
%  +(\nabla\,\xi_i)^\bot.
\end{eqnarray}
\end{proposition}

\begin{proof}
(i) The scalar product of \eqref{4.2} with $Y$,
% for $X,Y\in f(TM)$,
using \eqref{3.1AA}, gives
\begin{align}\label{4.3}
 g((\pounds_{\xi_i}{f})X,Y) &= N^{\,(5)}(\xi_i, X, Y) + g({f}\nabla_{X}\,\xi_i - \nabla_{{f} X}\,\xi_i,\ Y).
\end{align}
Similarly,
\begin{align}\label{4.3b}
 g((\pounds_{\xi_i}{f})Y,X) &=
 N^{\,(5)}(\xi_i, Y, X) +g({f}\nabla_{Y}\,\xi_i - \nabla_{{f} Y}\,\xi_i,\ X).
\end{align}
Using \eqref{2.7X} and $(fX)(\eta^i(Y)) -(fY)(\eta^i(X))\equiv0$
(this vanishes if either $X$ or $Y$ equals $\xi_j$ and also for $X$ and $Y$ in~$f(TM)$), we get
%\[
 $N^{\,(2)}_i (X,Y) = \eta^i([f Y, X]-[f X, Y])$.
%\]
Thus, the difference of \eqref{4.3} and \eqref{4.3b} gives
\begin{align*}
 2\,g((h_i-h_i^*)X,Y) =  N^{\,(5)}(\xi_i, X, Y) - N^{\,(2)}_i (X,Y).
\end{align*}
From this and equality $N^{\,(2)}_i=0$ (see Theorem~\ref{thm6.2}) we get \eqref{E-31}.
%%%%%%%%%%%%%%%%%%%%%%%%%%%%%%%%%%%%%%%%%%%%%%%%%

(ii) From Corollary \ref{cor3.1} with $Y=\xi_i$, we find
\begin{align}\label{4.4}
 g((\nabla_{X}{f})\xi_i,Z) &= -\frac12\,g(N^{\,(1)}(\xi_i,Z),{f} X)
 %- d\eta^i({f} Z,X)
 -g({f} X, {f} Z)
 + \frac12\,N^{\,(5)}(X,\xi_i,Z).
\end{align}
%% NEW
Note that $\frac 12\,N^{\,(5)}(X,\xi_i,Z)= g(h_i Z, \widetilde Q X)$, see \eqref{KK}.
By \eqref{2.5} with $Y=\xi_i$, we get
\begin{align}\label{2.5B}
 [{f},{f}](X,\xi_i) = {f}^2 [X,\xi_i] - {f}[{f} X,\xi_i] = f N^{\,(3)}_i (X).
\end{align}
Using \eqref{2.2}, \eqref{3.3B} and \eqref{2.5B}, we calculate
\begin{align}
\label{4.4A}
 g([{f},{f}](\xi_i,Z),{f} X)&= g({f}^2\,[\xi_i,Z] - {f}[\xi_i,{f} Z],{f} X) = - g({f}(\pounds_{\xi_i}{f})Z,{f} X)\notag\\
 &= g((\pounds_{\xi_i}{f})Z,QX) -\sum\nolimits_{j}\eta^j(X)\,\eta^j((\pounds_{\xi_i}{f})Z) .
 \end{align}
%%%%%%%%%%%%%%%%%%For a weak almost para-${\mathcal S}$-structure,
From \eqref{2.3} we have
$g([X,\xi_i], \xi_k) = 2\,d\eta^k(\xi_i,X)=2\,\Phi(\xi_i,X)=0$.
By \eqref{6.1e}, we get
$g(\nabla_X\,\xi_i, \xi_k) = g(\nabla_{\xi_i}X, \xi_k) = -g(\nabla_{\xi_i}\xi_k, X) = 0$ for $X\in f(TM)$,
thus
\begin{align}\label{E-30-xi}
  g(\nabla_{X}\,\xi_i,\ \xi_k) = 0,\quad X\in TM,\ 1\le i,k \le p .
\end{align}
Using \eqref{4.2}, we get
\begin{align}\label{3.1A3}
 2\,g((\nabla_{\xi_i}{f})Y,\xi_j) \overset{\eqref{3.1AA}}= N^{\,(5)}(\xi_i,Y,\xi_j) \overset{\eqref{KK}}=0 .
\end{align}
%%%%%%%%%%%%%%%%%%%%%%% MOVED
From \eqref{4.2}, \eqref{E-30-xi} and \eqref{3.1A3} we get
\begin{align}\label{4.5}
 g((\pounds_{\xi_i}{f})X,\xi_j) = -g(\nabla_{{f} X}\,\xi_i,\xi_j) = 0.
\end{align}
Since ${f}\,\xi_i=0$, we find
\begin{align}\label{4.5A}
 (\nabla_{X}{f})\,\xi_i = -{f}\,\nabla_{X}\,\xi_i.
\end{align}
Thus, combining \eqref{4.4}, \eqref{4.4A} and \eqref{4.5}, we find
\begin{eqnarray}\label{4.6}
\nonumber
 & -g({f}\,\nabla_{X}\xi_i, Z) = g(X,QZ) - g(h_iZ,QX)  - \sum\nolimits_{j}\eta^j(X)\eta^j(Z) + g(h_i Z, \widetilde Q X) \\
 & = g(h_iZ, X) + g(X,QZ) - \sum\nolimits_{j}\eta^j(X)\,\eta^j(Z) + g(h_i Z, \widetilde Q X) .
% \frac 12\,N^{\,(5)}(X,\xi_i,Z).
\end{eqnarray}
Replacing $Z$ by ${f} Z$ in \eqref{4.6} and using \eqref{2.1}, \eqref{E-30-xi} and ${f}\,\xi_i=0$, we achieve \eqref{E-30b}:
\begin{align*}
  g(Q\,\nabla_{X}\,\xi_i, Z) = g(({f} Q -h_i{f}) Z, X)
 % + g(f Z, \widetilde Q X)   \\  &
 = g( {f}( h^*_i - Q) X, Z) .
%   + g(f \widetilde Q X, Z),
\end{align*}

(iii) Using \eqref{2.1}, we obtain
\begin{eqnarray*}
 & {f}\nabla_{\xi_i}{f} +(\nabla_{\xi_i}{f}){f} = \nabla_{\xi_i}\,({f}^2)
 = \nabla_{\xi_i}\widetilde{Q} -\nabla_{\xi_i}(\sum\nolimits_{j}\eta^j\otimes \xi_j) ,
\end{eqnarray*}
where in view of \eqref{6.1e}, we get $\nabla_{\xi_i}(\sum\nolimits_{j}\eta^j\otimes \xi_j)=0$.
From the above and \eqref{4.2}, we get \eqref{E-31A}:
\begin{eqnarray*}
 && 2(h_i{f}+{f} h_i)X  = {f}(\pounds_{\xi_i}{f})X +(\pounds_{\xi_i}{f}){f} X \\
 && = {f}(\nabla_{\xi_i}{f})X +(\nabla_{\xi_i}{f}){f} X +{f}^2\nabla_X\,\xi_i -\nabla_{{f}^2 X}\,\xi_i \\
 && = -(\nabla_{\xi_i}\widetilde{Q})X -\widetilde{Q}\nabla_X\,\xi_i+\nabla_{\widetilde{Q}X}\,\xi_i
 +\sum\nolimits_{j}\big(g(\nabla_X\,\xi_i, \xi_j)\,\xi_j -g(X, \xi_j)\nabla_{\xi_j}\,\xi_i\big) \\
 && = [\widetilde{Q}X, \xi_i] - \widetilde{Q}\,[X, \xi_i]
% + (\nabla_X\,\xi_i)^\bot
  = -(\pounds_{\xi_i}\widetilde{Q})X .
% = -(\pounds_{\xi_i}{Q})X .
\end{eqnarray*}
%Using \eqref{6.1e} and  \eqref{E-30-xi} in \eqref{E-30} yields~\eqref{E-30b}.
%
%Here,
We used \eqref{6.1e} and
%the property $((h_i{f}+{f} h_i)X)^\bot=0$.
\eqref{E-30-xi}
to show $\sum\nolimits_{j}\big(g(\nabla_X\,\xi_i, \xi_j)\,\xi_j -g(X, \xi_j)\nabla_{\xi_j}\,\xi_i\big)=0$.
%$\sum\nolimits_{j}\,(g(\nabla_X\,\xi_i, \xi_j)\,\xi_j -g(X, \xi_j)\nabla_{\xi_j}\,\xi_i\,)=0$.
\end{proof}

\begin{remark}\rm
For a weak almost para-${\mathcal S}$-structure, using \eqref{3.1A3}, we find
%the following:
\[
 2\,g(h_i X, \xi_j)=-g(\nabla_{f X}\,\xi_i, \xi_j) \overset{\eqref{E-30-xi}}= 0;
\]
thus,
the distribution
$f(TM)$ is invariant under $h_i$; moreover, $h^*_i\,\xi_j = 0$, see also \eqref{4.2b}.
\end{remark}

The next statement follows from Propositions~\ref{thm6.1} and \ref{lem6.1}.
% generalizes ??
%\cite[Proposition 1.3]{b1970} with $Q={\rm id}_{TM}$.

\begin{corollary}
%\label{cor3.1K}
For a weak para-${\mathcal K}$-structure, we have
% the covariant derivative of ${f}$ is given~by
\begin{eqnarray}\label{3.1K}
\nonumber
 && 2\,g((\nabla_{X}{f})Y,Z) =
 \sum\nolimits_{i}\big( 2\,d\eta^i({f} Y,X)\,\eta^i(Z) - 2\,d\eta^i({f} Z,X)\,\eta^i(Y) \\
 && +\,\eta^i([\widetilde{Q}Y,\,{f} Z])\,\eta^i(X) \big)
 + N^{\,(5)}(X,Y,Z).
%,\quad 1\le i\le p,
\end{eqnarray}
In particular, using \eqref{E-31} with $h_i=0$, gives
 $2\,g((\nabla_{\xi_i}{f})Y,Z) = \eta^i([\widetilde{Q} Y,\,{f} Z])$ for $1\le i\le p$.
\end{corollary}

%\begin{proof}
% This follows directly from \eqref{3.1} with $N^{\,(1)}=0$ and \eqref{3.1KK}.
%\end{proof}

%\begin{remark}\rm
%One can rewrite \eqref{3.1}, \eqref{3.1A} and \eqref{3.1K} in terms of~$\Phi$, using
%the equality
%\begin{equation}\label{E-nabla-Phi}
%(\nabla_{X}\Phi)(Z,Y)=g((\nabla_{X}{f})Y,Z).
%\end{equation}
%\end{remark}

\section{The rigidity of a para-${\mathcal S}$-structure}
\label{sec:3}

An important class of metric para-$f$-manifolds is given by para-${\mathcal S}$-manifolds.
Here, we study a wider class of weak para-${\mathcal S}$-manifolds and prove the rigidity theorem for para-${\mathcal S}$-manifolds.
%in the class of metric weak para-$f$-manifolds.
%The following result generalizes ??
%\cite[Proposition~1.4]{b1970} and \cite[Proposition~3.11]{BP-2008A}.

\begin{proposition}
%\label{thm5.1}
For a weak para-${\mathcal S}$-structure we get
\begin{eqnarray}\label{4.10} %para-OK
\nonumber
 & g((\nabla_{X}{f})Y,Z)  = g(QX,Z)\,\bar\eta(Y) - g(QX,Y)\,\bar\eta(Z) +\frac{1}{2}\, N^{\,(5)}(X,Y,Z) \\
 & -\sum\nolimits_{j} \eta^j(X)\big(\bar\eta(Y)\eta^j(Z) - \eta^j(Y)\bar\eta(Z)\big) .
\end{eqnarray}
\end{proposition}

\begin{proof} Since $({f},Q,\xi_i,\eta^i,g)$ is a metric
weak $f$-structure with $N^{\,(1)}=0$, by Corollary~\ref{cor3.1}, we get~\eqref{4.10}.
%Using \eqref{2.2} in the scalar product of \eqref{4.11} with $Z$, we get~\eqref{4.10}.
%%%%%%%%%
\end{proof}

\begin{remark}\rm
Using $Y=\xi_i$ in \eqref{4.10}, we get $f\nabla_{X}\,\xi_i = -f^2X - \frac{1}{2}\,(N^{\,(5)}(X,\xi_i,\,\cdot))^\flat$,
which gene\-ralizes the equality  $\nabla_{X}\,\xi_i=-fX$ for a para-${\mathcal S}$-structure, e.g., \cite{FP-2017}.
%
%Using the equality \eqref{E-nabla-Phi}, in \eqref{4.10}, we can write the condition of a weak para-${\mathcal S}$-structure in terms of~$\Phi$.
\end{remark}

It was shown in \cite{RWo-2} that
a weak almost para-${\mathcal S}$-structure with positive partial Ricci curvature can be deformed to an almost para-${\mathcal S}$-structure.
The~main result in this section is the following rigidity theorem.
% of a para-${\mathcal S}$-structure.

\begin{theorem}\label{T-4.1}
A metric weak para-$f$-structure
%$({f},Q,\xi_i,\eta^i,g)$
is a weak para-${\mathcal S}$-structure if and only if it is a para-${\mathcal S}$-structure.
\end{theorem}

\begin{proof}
Let $({f},Q,\xi_i,\eta^i,g)$ be a weak para-${\mathcal S}$-structure.
Since $N^{\,(1)}=0$, by Proposition~\ref{thm6.1}, we get $N^{\,(3)}_i=0$.
By \eqref{KK}, we then obtain $N^{\,(5)}(\cdot\,,\xi_i,\,\cdot\,)=0$.
Recall that $\tilde{Q}X = Q X - X$ and $\eta^j(\widetilde{Q} X)=0$.
Using the above  and $Y=\xi_i$ in \eqref{4.10}, we~get
\begin{align}\label{4.12}
\nonumber
 & g((\nabla_{X}{f})\,\xi_i,Z)  = g(QX,Z) -\eta^i(Q X)\,\bar\eta(Z)
 +\sum\nolimits_{j} \eta^j(X)\big(\eta^j(Z) - \delta^j_i\,\bar\eta(Z)\big)\\
\nonumber
 & = g(Q X^\top, Z) +\sum\nolimits_{j} \eta^j(Z)\big(\eta^j(Q X) - \eta^i(Q X)\big)
-\sum\nolimits_{j} \eta^j(Z)\big(\eta^j(X) - \eta^i(X)\big)\\
 & = g(Q X^\top, Z) +\sum\nolimits_{j} \eta^j(Z)\big(\eta^j(\widetilde{Q} X) - \eta^i(\widetilde{Q} X)\big)
 = g(Q X^\top, Z) .
\end{align}
Using \eqref{4.5A}, we rewrite \eqref{4.12} as $g(\nabla_{X}\,\xi_i,{f} Z) = g(Q X^\top,Z)$.
By the above and \eqref{2.1}, we find
\begin{align}\label{4.14}
 g(\nabla_X\,\xi_i +{f} X^\top, \,{f}\,Z) = 0.
\end{align}
%Take $\tilde X,\tilde Y\in f(TM)$ such that ${f}\tilde X = X^\top$ and ${f}\,\tilde Y=Y^\top$.
%%%%%%%%%%%%%%%%%
Since ${f}$ is skew-symmetric, applying \eqref{4.10} with $Z=\xi_i$ in \eqref{4.NN}, we obtain
\begin{eqnarray}\label{4.17}
&& g( [{f},{f}](X,Y),\xi_i) = g([{f} X, {f} Y], \xi_i) = g((\nabla_{{f} X}{f})Y, \xi_i) - g((\nabla_{{f} Y}{f})X, \xi_i)   \notag\\
&&\quad = g(Q\,{f} Y,X) - g(Q\,{f} Y, \xi_i)\,\bar\eta(X) -g(Q\,{f} X,Y) +g(Q\,{f} X, \xi_i)\,\bar\eta(Y) .
\end{eqnarray}
Recall that $[Q,\,{f}]=0$ and $f\,\xi_i=0$. Thus, \eqref{4.17} yields for all $i$,
\begin{align*}
  g( [{f},{f}](X,Y), \xi_i) = 2\,g(QX,{f} Y) .
\end{align*}
From this, using the definition of $N^{\,(1)}$, we get for all $i$,
\begin{align}\label{4.18}
 g(N^{\,(1)}(X,Y), \xi_i)  = 2\,g(\widetilde{Q} X, {f} Y) .
\end{align}
From $N^{\,(1)}=0$ and \eqref{4.18} we get $g(\widetilde{Q} X, {f} Y)=0$ for all $X,Y\in \mathfrak{X}_M$; thus, $\widetilde Q=0$.
\end{proof}

For a weak almost para-${\mathcal S}$-structure all $\xi_i$ are Killing if and only if $h=0$, see Theorem~\ref{thm6.2}.
The equality $h=0$ holds for a weak para-${\mathcal S}$-structure since it is true for a para-${\mathcal S}$-structure, see Theorem~\ref{T-4.1}.
We~will prove this property of a weak para-${\mathcal S}$-structure directly.

\begin{corollary}
%\label{P-4.1}
For a weak para-${\mathcal S}$-structure, $\xi_1,\ldots,\xi_p$ are Killing vector fields; moreover, $\ker f$
is integrable and
defines a Riemannian totally geodesic foliation.
\end{corollary}

\begin{proof}
In view of \eqref{4.5A} and $\bar\eta(\xi_i)=1$, Eq. \eqref{4.10} with $Y=\xi_i$ becomes
\begin{align}\label{4.6A}
 g(\nabla_{X}\,\xi_i, {f} Z) = -\eta^i(X)\,\bar\eta(Z) +g(X,QZ) + \frac 12\,N^{\,(5)}(X,\xi_i,Z) .
\end{align}
Combining \eqref{4.6} and \eqref{4.6A}, and using \eqref{E-30-xi},
we achieve for all $i$ and $X,Z$,
\begin{align*}
 g(h_iZ,QX) = \sum\nolimits_{j}\eta^j(X)\,\eta^j(Z) -\eta^i(X)\,\bar\eta(Z),
% -\frac12\sum\nolimits_{j}\eta^j(X)\,g(\nabla_{fZ}\,\xi_i, \xi_j),
\end{align*}
which implies $hZ=0$ for $Z\in f(TM)$ (since $Q$ is nonsingular).
This and \eqref{4.2b} yield $h=0$.
By~Theorem~\ref{thm6.2}, $\ker f$ defines a totally geodesic foliation. Since $\xi_i$ is a Killing field,
we~get
\[
 0 = (\pounds_{\xi_i}\,g)(X,Y) = g(\nabla_{X}\,\xi_i, Y) + g(\nabla_{Y}\,\xi_i, X)
 = -g(\nabla_{X} Y + \nabla_{Y} X,\ \xi_i)
\]
for all $i$ and $X,Y\bot\,\ker f$. Thus, $f(TM)$ is totally geodesic, i.e., $\ker f$ defines a Riemannian foliation.
\end{proof}

For $p=1$, from Theorem~\ref{T-4.1} we have the following

\begin{corollary}
%[see \cite{RovP-arxiv}]
%\label{T-4.1}
A weak almost paracontact metric structure on $M^{2n+1}$ is a weak para-Sasakian structure if and only if it is a para-Sasakian structure,
i.e., a normal weak paracontact metric structure, on $M^{2n+1}$.
\end{corollary}

\section{The characteristic of a weak para-${\mathcal C}$-structure}
\label{sec:4}

An important class of metric para-$f$-manifolds is given by para-${\mathcal C}$-mani\-folds.
Recall that $\nabla_{X}\,\xi_i=0$ holds on para-${\mathcal C}$-manifolds.

\begin{proposition}
Let $({f},Q,\xi_i,\eta^i,g)$ be a weak
%almost
para-${\mathcal C}$-structure. Then
\begin{align}\label{6.1}
 & 2\,g((\nabla_{X}{f})Y,Z) = N^{\,(5)}(X,Y,Z),\\
\label{6.1b}
% 6\,d\Phi(X,Y,Z)
 & 0 = N^{\,(5)}(X,Y,Z) + N^{\,(5)}(Y,Z,X) + N^{\,(5)}(Z, X, Y) ,\\
\label{6.1c}
%  2\,g([{f},{f}](X,Y),Z)
 0 & = N^{\,(5)}({f} X,Y,Z) + N^{\,(5)}({f} Y,Z,X) + N^{\,(5)}({f} Z, X, Y) .
\end{align}
%In particular,
Using \eqref{6.1} with $Y=\xi_i$ and \eqref{2.1}, we get
\begin{align*}
%\label{6.2}
 g(\nabla_{X}\,\xi_i,\,Q Z) = -\frac12\,N^{\,(5)}(X,\xi_i,{f} Z).
\end{align*}
\end{proposition}

\begin{proof}
For a weak almost para-${\mathcal C}$-structure $({f},Q,\xi_i,\eta^i,g)$, using Theorem~\ref{thm6.2C},  from \eqref{3.1}
%with $N^{\,(1)}=0$
we get
\begin{equation}\label{6.1a}
 2\,g((\nabla_{X}{f})Y,Z)= - g([{f},{f}](Y,Z),{f} X) + N^{\,(5)}(X,Y,Z).
\end{equation}
From \eqref{6.1a}, using condition $[{f},{f}]=0$ we get \eqref{6.1}.
Using \eqref{3.3} and \eqref{6.1}, we write
\[
  0 = 3\,d\Phi(X,Y,Z) = g((\nabla_X\,{f})Z,Y) +g((\nabla_Y\,{f})X, Z) +g((\nabla_Z\,{f})Y, X);
\]
hence, \eqref{6.1b} is true.
Using \eqref{4.NN}, \eqref{6.1} and the skew-symmetry of ${f}$, we obtain
\begin{align*}
 0 & = 2\,g([{f},{f}](X,Y),Z) \\
 & = N^{\,(5)}(X, Y, {f} Z) + N^{\,(5)}({f} X, Y, Z)
  - N^{\,(5)}(Y, X, {f} Z) - N^{\,(5)}({f} Y,X,Z) .
\end{align*}
This and \eqref{6.1b} with $X$ replaced by ${f} X$ provide \eqref{6.1c}.
\end{proof}

Recall that $X^\bot = \sum\nolimits_{i}\eta^i(X)\,\xi_i$.
Consider a weaker condition than \eqref{6.1d}:
\begin{align}\label{E-xi31}
% g([\xi_i,\xi_j],\xi_k) =0,\quad 1\le i,j,k\le p.
 [\xi_i,\xi_j]^\bot =0,\quad 1\le i,j\le p.
\end{align}

%Recall that a para-${\mathcal K}$-structure is a para-${\mathcal C}$-structure if and only if ${f}$ is parallel, e.g., ??
%%\cite[Theorem~1.5]{b1970}
%(and \cite[Theorem~6.8]{blair2010riemannian} for $p=1)$.
% The~following our theorem extends these results.
In the following theorem, we characterize weak para-${\mathcal C}$-manifolds in a wider class of metric weak para-$f$-manifolds
using the condition $\nabla{f}=0$.

\begin{theorem}\label{thm6.2D}
A metric weak para-$f$-structure with $\nabla{f}=0$ and \eqref{E-xi31} is a~weak para-${\mathcal C}$-structure with $N^{\,(5)}=0$.
\end{theorem}

\begin{proof}
Using condition $\nabla{f}=0$, from \eqref{4.NN} we obtain $[{f},{f}]=0$.
Hence, from \eqref{2.6X} we get $N^{\,(1)}(X,Y)=-2\,\sum\nolimits_{i} d\eta^i(X,Y)\,\xi_i$,
and from \eqref{4.NNxi} we obtain
\begin{align}\label{E-cond1}
 \nabla_{{f} X}\,\xi_i - {f}\,\nabla_{X}\,\xi_i = 0,\quad X\in \mathfrak{X}_M.
\end{align}
From \eqref{3.3}, we calculate
\[
 3\,d\Phi(X,Y,Z) = g((\nabla_{X}{f})Z, Y) + g((\nabla_{Y}{f})X,Z) + g((\nabla_{Z}{f})Y,X);
\]
hence, using condition $\nabla{f}=0$ again, we get $d\Phi=0$. Next,
%\begin{align*}
 $N^{\,(2)}_i(Y,\xi_j) = -\eta^i([{f} Y,\xi_j]) = g(\xi_j, {f}\nabla_{\xi_i} Y) =0$.
%\end{align*}
%Thus,
Setting $Z=\xi_j$ in \eqref{3.1} and using the condition $\nabla{f}=0$ and the properties
$d\Phi=0$, $N^{\,(2)}_i(Y,\xi_j)=0$ and $N^{\,(1)}(X,Y)=-2\sum\nolimits_{i} d\eta^i(X,Y)\,\xi_i$, we find
%\[
 $0 = 2\,d\eta^j({f} Y, X) - N^{\,(5)}(X,\xi_j, Y)$.
%\]
 By~\eqref{KK} and~\eqref{E-cond1},
 %we get
\[
 N^{\,(5)}(X,\xi_j, Y) = g([\xi_j,{f} Y] -{f}[\xi_j,Y],\, \widetilde{Q} X)
 = g(\nabla_{{f} Y}\,\xi_j - {f}\,\nabla_{Y}\,\xi_j,\, \widetilde{Q} X) = 0;
\]
hence, $d\eta^j({f} Y, X)=0$. From this and $g([\xi_i,\xi_j],\xi_k)=2\,d\eta^k(\xi_j, \xi_i)=0$ we get $d\eta^j=0$.
By the above, $N^{\,(1)}=0$. Thus, $({f},Q,\xi_i,\eta^i,g)$ is a weak para-${\mathcal C}$-structure.
Finally, from \eqref{6.1} and condition $\nabla{f}=0$ we get $N^{\,(5)}=0$.
\end{proof}

\begin{corollary}
A normal metric weak para-$f$-structure with
%the property
$\nabla f=0$
is a~weak para-${\mathcal C}$-structure with
%vanishing tensor
$N^{\,(5)}=0$.
\end{corollary}

\begin{proof} By $N^{\,(1)}=0$, we get $d\eta^i =0$ for all $i$.
As in Theorem~\ref{thm6.2D}, we get $d\Phi=0$.
\end{proof}

\begin{example}\rm
Let $M$ be a $2n$-dimensional smooth manifold and $\tilde{f}:TM\to TM$ an endomorphism of rank $2n$ such that
$\nabla\tilde{f}=0$.
To construct a weak para-${\mathcal C}$-structure on $M\times\mathbb{R}^p$
(or $M\times \mathbb{T}^p$, where
%$\mathbb{R}^p$ is a Euclidean space and
$\mathbb{T}^p$ is a $p$-dimensional flat torus),
take any point $(x, t_1,\ldots,t_p)$
%of either space
and set $\xi_i = (0, d/dt_i)$, $\eta^i =(0, dt_i)$~and
\[
 {f}(X, Y) = (\tilde{f} X,\, 0),\quad
 Q(X, Y) = (\tilde{f}^{\,2} X,\, Y).
\]
where $X\in T_xM$ and $Y=\sum_i Y^i\xi_i\in\{\mathbb{R}^p_t, \mathbb{T}^p_t\}$.
Then \eqref{2.1} holds and Theorem~\ref{thm6.2D} can be used.
\end{example}

For $p=1$, from Theorem~\ref{thm6.2D} we have the following

\begin{corollary}
%[see \cite{RovP-arxiv}]
Any weak almost paracontact structure $(\varphi,Q,\xi,\eta,g)$ with the property $\nabla\varphi=0$
is a~weak para-cosymplectic structure.
%, i.e., $d\Phi=0$ and $d\eta=0$, with vanishing tensor $N^{\,(5)}$.
\end{corollary}
%%%%%%%%%%%%%%%%%%%%%%%%%

%\setcounter{equation}{0}
%%%%%%%%%%
%\section{Conclusions}
%We introduced and studied a new class of self-concordant functions,
%defined on Riemannian manifolds endowed with metrics of diagonal type ...

%\bigskip

%%%%%%%%%%%

%
%-------------------------------------------------------------------
%

%-------------------------------------------------------------------
%

\begin{thebibliography}{99}
\setlength{\baselineskip}{.45cm}

%author name in italic

%volume number boldface

%standard abreviation is necessary

\bibitem{AM-1995}
{\it D. Alekseevsky and P. Michor}, {Differential geometry of $\mathfrak{g}$-manifolds}, Differential Geom. Appl. {\bf 5}(1995), 371-403.

%\bibitem{balkan}
%{\it Y.\,S. Balkan et al.}, A new class of $f$-structures satisfying $f^3-f=0$. Filomat {\bf 32}(2018), No. 17, 5919-5929.

\bibitem{BL-69}
{\it D.\,E. Blair and G.\,D. Ludden}, Hypersurfaces in almost contact manifolds. Tohoku Math. J. {\bf 21}(1969), 354-362.

\bibitem{BN-1985}
{\it A. Bucki and A. Miernowski}, Almost $r$-paracontact structures. Ann. Univ. Mariae Curie-Sklodowska {\bf 39}(1985), No. 2, 13-26.

\bibitem{CFG-survey}
{\it V. Cruceanu, P. Fortuny and P.\,M. Gadea}, A survey on paracontact geometry. Rocky Mt. J. Math. {\bf 26}(1996), No. 1, 83-115.

\bibitem{fip}
{\it M. Falcitelli, S. Ianus and A.\,M. Pastore}, Riemannian Submersions and Related Topics, World Scientific, 2004.

\bibitem{FP-2017}
{\it L.\,M. Fern\'{a}ndez and A. Prieto-Mart\'{i}n}, On $\eta$-Einstein para-$S$-manifolds.
Bull. Malays. Math. Sci. Soc. {\bf 40}(2017), 1623-1637.

\bibitem{g1967}
{\it A. Gray}, {Pseudo-Riemannian almost product manifolds and submersions}, J. Math. Mech., {\bf 16}(1967), No. 7, 715-737.

\bibitem{KN-69}
{\it S. Kobayashi and K. Nomizu}, \textit{Foundations of differential geometry}, Vols. I, II,
USA, Interscience Publishers, New York--London--Sydney, 1963, 1969.

\bibitem{m1976}
{\it K. Matsumoto}, On a structure defined by a tensor field $f$ of type (1,1) satisfying $f^3-f=0$, Bull. Yamagata Univ. {\bf 1}(1976), 33-47.

\bibitem{Rov-arxiv}
{\it V.\,Rovenski}, {On the geometry of a weakened $f$-structure}, arXiv:2205.02158, 2022, 14\,pp.

\bibitem{RWo-2}
{\it V. Rovenski and R. Wolak}, {New metric structures on $\mathfrak{g}$-foliations}, Indagationes Mathema\-ticae, {\bf 33}(2022), 518-532.

%\bibitem{sato}
%{\it I. Sato}, On a structure similar to the almost contact structure. Tensor (NS) {\bf 30}(1976), No. 3, 219-224.

%\bibitem{yan}
%{\it K. Yano}, On a structure $f$ satisfying $f^3+f=0$, Technical Report No. 12, University of Washington, 1961.

\end{thebibliography}
\end{document}